\documentclass[12pt]{amsart}
\usepackage{amscd,amssymb,amsthm,amsmath,amssymb,textcomp,supertabular,longtable,enumerate,rotating,mathrsfs,mathtools,wasysym}
\usepackage[matrix,arrow,curve]{xy}
\usepackage{pdflscape}

\newtheorem*{theorem}{Theorem}
\newtheorem*{claim}{Claim}
\newtheorem*{proposition}{Proposition}
\newtheorem*{lemma}{Lemma}
\newtheorem*{corollary}{Corollary}
\theoremstyle{definition}
\newtheorem*{example}{Example}
\theoremstyle{remark}
\newtheorem*{remark}{Remark}

\newcommand{\mumu}{\boldsymbol{\mu}}

\usepackage{xcolor}

\newcommand{\blue}[1]{\leavevmode{\color{blue}{#1}}}

\makeatletter
\newcommand{\pushright}[1]{\ifmeasuring@#1\else\omit\hfill$\displaystyle#1$\fi\ignorespaces}
\newcommand{\pushleft}[1]{\ifmeasuring@#1\else\omit$\displaystyle#1$\hfill\fi\ignorespaces}
\makeatother

\author{Ivan Cheltsov and Jihun Park}

\title{K-stable Fano threefolds of rank $2$ and degree $30$}

\address{ \emph{Ivan Cheltsov}\newline \textnormal{School of Mathematics, University of Edinburgh, Edinburgh, Scotland
\newline
\texttt{I.Cheltsov@ed.ac.uk}}}

\address{ \emph{Jihun Park}\newline \textnormal{Center for Geometry and Physics, Institute for Basic Science, Pohang, Korea \newline
Department of Mathematics, POSTECH, Pohang, Korea \newline
\texttt{wlog@postech.ac.kr}}}

\begin{document}

\begin{abstract}
We find all K-stable smooth Fano threefolds in the family \textnumero 2.22.
\end{abstract}

\maketitle

Let $X$ be a~smooth Fano threefold. Then $X$ belongs to one of the $105$ families,
which are labeled as \textnumero 1.1, \textnumero 1.2, $\ldots$, \textnumero 9.1, \textnumero 10.1.
See \cite{ACCFKMGSSV}, for the description of these families.
If $X$ is a~general member of the~family \textnumero $\mathscr{N}$, then \cite[Main~Theorem]{ACCFKMGSSV} gives
$$
X\ \text{is K-polystable} \iff\mathscr{N}\not\in\left\{\aligned
&\ 2.23, 2.26, 2.28, 2.30, 2.31, 2.33, 2.35, 2.36, \\
& \ 3.14, 3.16, 3.18,3.21, 3.22, 3.23, \\
&\ 3.24, 3.26, 3.28, 3.29,  3.30, 3.31, \\
&\    4.5, 4.8, 4.9, 4.10, 4.11, 4.12,\\
&\  5.2
\endaligned
\right\}.
$$
The goal of this note is to find all K-polystable smooth Fano threefolds in the family \textnumero 2.22.
This family contains both K-polystable and non-K-polystable smooth Fano threefolds,
and a~conjectural characterization of all K-polystable members has been given in \cite[\S~7.4]{ACCFKMGSSV}.
We will confirm this conjecture --- this will complete the~description of all K-polystable smooth Fano threefolds of Picard rank $2$ and degree $30$ started in \cite{ACCFKMGSSV}.

Starting from now, we suppose that $X$ is a~smooth Fano threefold in the~family  \textnumero 2.22.
Then $X$ can be described both as the~blow up of $\mathbb{P}^3$ along a~smooth twisted quartic curve, and
the blow up of $V_5$, the~unique smooth threefold \textnumero 1.15, along an~irreducible conic.
More precisely, there are a~smooth twisted quartic curve $C_4\subset\mathbb{P}^3$, a~smooth conic $C\subset V_5$,
and a~commutative diagram
$$
\xymatrix{
&X\ar@{->}[ld]_{\pi}\ar@{->}[rd]^{\phi}&\\%
\mathbb{P}^3\ar@{-->}[rr]^{\psi}&&V_5,}
$$
where $\pi$ is the~blow up of $\mathbb{P}^3$ along $C_4$,
$\phi$ is~the~blow up of  $V_5$ along $C$,
and $\psi$ is given by the~linear system of cubic surfaces  containing~$C_4$. Here, $V_5$ is  embedded in $\mathbb{P}^6$ as described in~\mbox{\cite[\S~5.10]{ACCFKMGSSV}}.
All smooth Fano threefolds in the~family  \textnumero 2.22 can be obtained in this way.

The~curve $C_4$ is contained in~a~unique smooth quadric surface $Q\subset\mathbb{P}^3$, and $\phi$ contracts the~proper transform of this surface.
Note that 
$$
\mathrm{Aut}(X)\cong\mathrm{Aut}(\mathbb{P}^3,C_4)\cong\mathrm{Aut}(Q,C_4).
$$
Choosing appropriate coordinates on $\mathbb{P}^3$, we may assume that $Q$ is given by $x_0x_3=x_1x_2$, where $[x_0: x_1: x_2: x_3]$ are  coordinates on $\mathbb{P}^3$.
Fix the~isomorphism $Q\cong\mathbb{P}^1\times\mathbb{P}^1$ given by
$$
\big([u:v],[x:y]\big)\mapsto\big[xu:xv:yu:yv\big],
$$
where $([u:v],[x:y])$ are coordinates in $\mathbb{P}^1\times\mathbb{P}^1$.
Swapping $[u:v]$ and $[x:y]$ if necessary, we may assume that $C_4$ is a~divisor of degree $(1,3)$ in $Q$,
so that $C_4$ is given in $Q$ by
$$
uf_3(x,y)=vg_3(x,y)
$$
for some non-zero cubic homogeneous polynomials $f_3(x,y)$ and $g_3(x,y)$.

Let $\sigma\colon C_4\to\mathbb{P}^1$ be the map given by the projection $([u:v],[x:y])\mapsto [u:v]$.
Then $\sigma$ is a~triple cover, which is ramified over at least two points.
After an~appropriate change of coordinates~$[u:v]$, we may assume that $\sigma$ is ramified over $[1:0]$ and $[0:1]$.
Then both $f_3$ and $g_3$ have multiple zeros in $\mathbb{P}^1$.
Changing coordinates $[x:y]$, we may assume that these zeros are $[0:1]$ and $[1:0]$, respectively.
Keeping in mind that the~curve $C_4$ is smooth, we see that $C_4$ is given by
$$
u(x^3+ax^2y)=v(y^3+by^2x)
$$
for some complex numbers $a$ and $b$, after a~suitable scaling of the~coordinates.
If $a=b=0$, then the~curve $C_4$ is given by $ux^3=vy^3$,
which gives $\mathrm{Aut}(X)\cong\mathrm{Aut}(Q,C_4)\cong\mathbb{G}_m\rtimes\mumu_2$.
In this case, the threefold $X$ is known to be K-polystable \cite[\S~4.4]{ACCFKMGSSV}.

\begin{example}
Suppose that $ab=0$, but $a\ne 0$ or $b\ne 0$.
We can scale the~coordinates further and swap
them if necessary, and assume that the~curve $C_4$ is given by 
$$
ux^3=v(y^3+y^2x).
$$
In this case, the threefold $X$ is not K-polystable \cite[\S~7.4]{ACCFKMGSSV}.
\end{example}

A conjecture in \cite[\S~7.4]{ACCFKMGSSV} says that the non-K-polystable Fano threefold described in this example is
the~unique non-K-polystable smooth Fano threefold in the~family~\textnumero 2.22.
Let us show that this is indeed the case.
To do this, we may assume that $a\ne 0$ and $b\ne 0$. Then, scaling the~coordinates, we~may assume that $C_4$ is given by
\begin{equation}
\label{equation:curve}\tag{$\bigstar$}
u\big(x^3+\lambda x^2y\big)=v\big(y^3+\lambda y^2x\big)
\end{equation}
for some non-zero complex number $\lambda$.
Since the curve $C_4$ is smooth, we must have $\lambda\ne\pm 1$.
Moreover, if $\lambda=\pm 3$, then we can change the~coordinates on $Q$ in such a way  that $C_4$ would be given by $ux^3=v(y^3+y^2x)$,
so that $X$ is not K-polystable in this case.

We know from \cite{ACCFKMGSSV} that $X$ is K-stable if $C_4$ is given by \eqref{equation:curve} with $\lambda$ general.
In particular, we know from \cite[\S~4.4]{ACCFKMGSSV} that the~threefold $X$ is K-stable when $\lambda=\pm\sqrt{3}$.
Our main result is the following theorem.

\begin{theorem}
Suppose that $C_4$ is given in \eqref{equation:curve} with $\lambda\not\in\{0,\pm 1,\pm 3\}$. Then $X$ is K-stable.
\end{theorem}

Let us prove this theorem.
We suppose that $C_4$ is given by \eqref{equation:curve} with $\lambda\not\in\{0,\pm 1,\pm 3\}$.
Then the~triple cover $\sigma\colon C_4\to\mathbb{P}^1$ is ramified in four distinct points $P_1$, $P_2$, $P_3$, $P_4$,
which implies that $\mathrm{Aut}(Q,C_4)$ is a finite group, since $\mathrm{Aut}(Q,C_4)\subset\mathrm{Aut}(C_4,P_1+P_2+P_3+P_4)$.
Without loss of generality, we may assume that
\begin{align*}
P_1&=\big([1:0],[0:1]\big)=[0:1:0:0]\\
P_2&=\big([0:1],[1:0]\big)=[0:0:1:0],
\end{align*}
where we use both the coordinates on $Q\cong\mathbb{P}^1\times\mathbb{P}^1$ and $\mathbb{P}^3$ simultaneously.

Observe that the~group $\mathrm{Aut}(Q,C_4)$ contains an~involution $\tau$ that is given by
$$
\big([u:v],[x:y]\big)\mapsto\big([v:u],[y:x]\big).
$$
Let us identify $\mathrm{Aut}(\mathbb{P}^3,C_4)=\mathrm{Aut}(Q,C_4)$ using the isomorphism $Q\cong\mathbb{P}^1\times\mathbb{P}^1$ fixed above.
Then $\tau$ is given by $[x_0:x_1:x_2:x_3]\mapsto[x_3:x_2:x_1:x_0]$.
Note that $\tau$ swaps $P_1$ and~$P_2$,
and the $\tau$-fixed points in $C_4$ are $([1:1],[1:1])$ and $([1:-1],[1:-1])$,
which are not ramification points of the triple cover $\sigma$.
This shows that $\tau$ swaps the~points $P_3$ and~$P_4$.
In~fact, the group $\mathrm{Aut}(Q,C_4)$ is larger than its subgroup $\langle\tau\rangle\cong\mumu_2$.
Indeed, one can change coordinates $([u:v],[x:y])$ on $Q$ such that
\begin{align*}
P_1&=([1:0],[0:1]),\\ 
P_4&=([0:1],[1:0]),
\end{align*} 
and the~curve $C_4$ is given by
$$
u\big(x^3+\lambda^\prime x^2y\big)=v\big(y^3+\lambda^\prime y^2x\big)
$$
for some complex number $\lambda^\prime\not\in\{0,\pm 1,\pm 3\}$.
This gives an~involution $\iota\in\mathrm{Aut}(Q,C_4)$ such that $\iota(P_1)=P_4$ and $\iota(P_2)=P_3$.
Let $G$ be the~subgroup $\langle\tau,\iota\rangle\subset\mathrm{Aut}(Q,C_4)=\mathrm{Aut}(\mathbb{P}^3,C_4)$.
Then $G\cong\mumu_2^2$.
Note that the group $\mathrm{Aut}(\mathbb{P}^3,C_4)$ can be larger for some $\lambda\in\mathbb{C}\setminus\{0,\pm 1,\pm 3\}$.
For instance, if $\lambda=\pm\sqrt{3}$, then $\mathrm{Aut}(\mathbb{P}^3,C_4)\cong\mathfrak{A}_4$, c.f. \cite[Example~4.4.6]{ACCFKMGSSV}.

The $G$-action on $C_4$ is faithful, so that the~curve $C_4$ does not contain $G$-fixed points.
Hence, the quadric $Q$ does not contain $G$-fixed points, since otherwise $Q$ would contain a $G$-invariant
curve of degree $(1,0)$, which would intersect $C_4$ by a $G$-fixed point.
So, in particular, we see that $\mathbb{P}^3$ contains finitely many $G$-fixed points.
Since the $G$-action on $\mathbb{P}^3$ is given by $4$-dimensional linear representation of the group $G$,
we conclude this representation splits as a sum of $4$ distinct one-dimensional representations,
which implies that the~space $\mathbb{P}^3$ contains exactly four $G$-fixed points.
Denote these points by $O_1$,~$O_2$,~$O_3$,~$O_4$.
These four points are not co-planar.
For every $1\leqslant i<j\leqslant 4$, let $L_{ij}$ be the line in $\mathbb{P}^3$ that passes through $O_i$ and $O_j$.
Then the lines $L_{12}$, $L_{13}$, $L_{14}$, $L_{23}$, $L_{24}$, $L_{34}$ are $G$-invariant,
and they are the only $G$-invariant lines in $\mathbb{P}^3$.
For each  $1\leqslant i\leqslant 4$, let $\Pi_i$ be the plane in $\mathbb{P}^3$ determined by the three points $\{O_1, O_2,O_3,O_4\}\setminus  \{O_i\}$.
Then the four planes $\Pi_1$, $\Pi_2$, $\Pi_3$, $\Pi_4$ are the only $G$-invariant planes in $\mathbb{P}^3$.

\begin{remark}
Each plane $\Pi_i$ intersects $C_4$ at four distinct points.
Indeed, if $|\Pi_i\cap C_4|<4$, then $\Pi_i\cap C_4$ is a $G$-orbit of length $2$, and $\Pi_i$ is tangent to $C_4$ at both the  points of this orbit.
Therefore, without loss of generality, we may assume that the~intersection $\Pi_i\cap C_4$ is just the~fixed locus of the~involution~$\tau$.
Then $\Pi_i\cap C_4=([1:1],[1:1])\cup([1:-1],[1:-1])$,
so~that $\Pi_i\vert_{Q}$ is a smooth conic that is given by 
$$
a(vx-uy)=b(ux-vy)
$$ 
for some~$[a:b]\in\mathbb{P}^1$.
But the~conic $\Pi_i\vert_{Q}$ cannot tangent $C_4$ at the~points $([1:1],[1:1])$ and $([1:-1],[1:-1])$,
so that $|\Pi_i\cap C_4|=4$. 
\end{remark}

The curve $C_4$ contains exactly three $G$-orbits of length $2$,
and these $G$-orbits are just the~fixed loci of the involutions $\tau$, $\iota$, $\tau\circ\iota$ described earlier.
Let~$L$,~$L^\prime$ and $L^{\prime\prime}$ be the three lines in $\mathbb{P}^3$ such that $L\cap C_4$, $L^{\prime}\cap C_4$ and $L^{\prime\prime}\cap C_4$ are
the~fixed loci of the involutions $\tau$, $\iota$ and $\tau\circ\iota$, respectively.
Then $L$, $L^\prime$ and $L^{\prime\prime}$  are $G$-invariant lines, so~that they are three lines among $L_{12}$, $L_{13}$, $L_{14}$, $L_{23}$, $L_{24}$, $L_{34}$.
In fact, it easily follows from Remark that the lines $L$, $L^\prime$, $L^{\prime\prime}$  meet at one point. Therefore, we may assume that
$L\cap L^{\prime}\cap L^{\prime\prime}=O_4$ and $L=L_{14}$, $L^{\prime}=L_{24}$, $L^{\prime\prime}=L_{34}$. Then
\begin{align*}
\Pi_1\cap C_4&=\big(L^{\prime}\cap C_4\big)\cup\big(L^{\prime\prime}\cap C_4\big), \\
\Pi_2\cap C_4&=\big(L\cap C_4\big)\cup\big(L^{\prime\prime}\cap C_4\big), \\
\Pi_3\cap C_4&=\big(L\cap C_4\big)\cup\big(L^{\prime}\cap C_4\big).
\end{align*}
On the other hand, the intersection $\Pi_4\cap C_4$ is a $G$-orbit of length $4$.

\bigskip

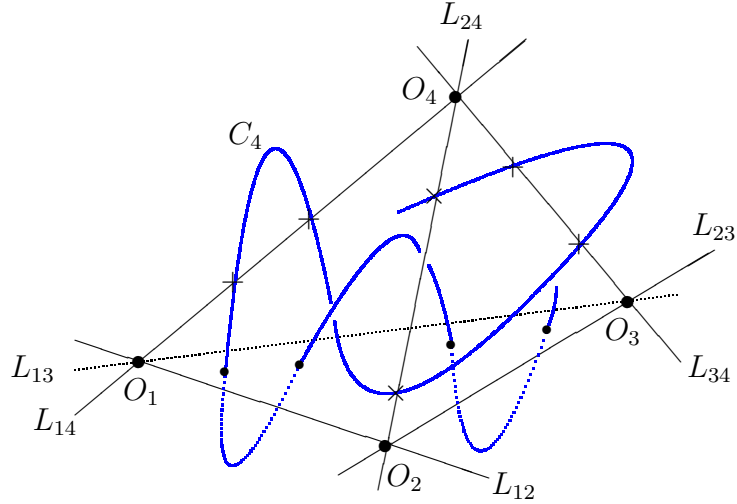
\begin{center}
\begin{minipage}[m]{.25\linewidth}\setlength{\unitlength}{1.00mm}
\begin{center}
\begin{picture}(100,70)(25,5)

\qbezier[150](10,26)(50,31)(90,36)

\put(10,20){\line(6,5){60}}
\put(90.5,27){\line(-5,6){35}}
\put(50.3,10){\line(1,5){12}}

\put(10,30){\line(3,-1){55}}
\put(45,12){\line(5,3){50}}

\thicklines

\put(64,9.5){\mbox{ $L_{12}$}}
\put(92,44){\mbox{$L_{23}$}}
\put(0,25){\mbox{ $L_{13}$}}

\put(3,18){\mbox{ $L_{14}$}}

\put(57,72){\mbox{ $L_{24}$}}

\put(90,25){\mbox{ $L_{34}$}}

\put(29,56){\mbox{ $C_4$}}

\blue{
\qbezier(53,47)(105,70)(70,35)

\qbezier(44.5,32.5)(47,12)(70,35)

\qbezier(44,35)(35,80)(30,26)

\qbezier[60](40,27)(27,0)(30,26)

\qbezier(40,27)(53,50)(55.8,41.7)

\qbezier(57,39.9)(58.5,37)(60,30)


\qbezier[60](73,32)(62,0)(60,29)

\qbezier(73,32)(74,35)(74,37)

}

\put(48.8,14.8){\mbox{ $\bullet$}}
\put(50,10.5){\mbox{ $O_2$}}

\put(58.1,61.2){\mbox{ $\bullet$}}
\put(52,62){\mbox{ $O_4$}}

\put(16,26){\mbox{ $\bullet$}}
\put(15,22){\mbox{ $O_1$}}

\put(81,34){\mbox{ $\bullet$}}
\put(79,30){\mbox{ $O_3$}}

\put(65.2,52){\mbox{ $+$}}
\put(54.7,48.1){\mbox{ $\times$}}

\put(38.1,45){\mbox{ $+$}}
\put(28,36.7){\mbox{ $+$}}

\put(49.6,21.9){\mbox{ $\times$}}
\put(74,41.7){\mbox{ $+$}}

\put(27.7,25){\mbox{ \tiny $\bullet$}}
\put(37.7,26){\mbox{ \tiny $\bullet$}}
\put(57.8,28.7){\mbox{ \tiny $\bullet$}}
\put(70.6,30.7){\mbox{ \tiny $\bullet$}}

\end{picture}
\end{center}
\end{minipage}
\end{center}

\bigskip

Since $C_4$ is $G$-invariant, the action of the group $G$ lifts to the~threefold $X$,
so that we also identify $G$ with a~subgroup of the~group $\mathrm{Aut}(X)$.
Let~$E$ be the~$\pi$-exceptional surface, let $\widetilde{Q}$ be the proper transform of the quadric $Q$ on the~threefold $X$,
let $H_1$, $H_2$, $H_3$ and $H_4$ be the proper transforms on $X$ of the $G$-invariant planes $\Pi_1$, $\Pi_2$, $\Pi_3$ and $\Pi_4$, respectively, and
let $H$ be the~proper transform on $X$ of a~general hyperplane in $\mathbb{P}^3$.
Then
$$
-K_X\sim 2\widetilde{Q}+E\sim \widetilde{Q}+2H_1\sim \widetilde{Q}+2H_2\sim \widetilde{Q}+2H_3\sim \widetilde{Q}+2H_4\sim 4H-E,
$$
and the surfaces $E$, $\widetilde{Q}$, $H_1$, $H_2$, $H_3$, $H_4$ are $G$-invariant.
Observe that $\widetilde{Q}\cong Q\cong\mathbb{P}^1\times\mathbb{P}^1$,
and $H_1$, $H_2$, $H_3$, $H_4$ are smooth del Pezzo surfaces of degree $5$.

\begin{claim}
Let $S$ be a possibly reducible $G$-invariant surface in $X$ such that $-K_X\sim_{\mathbb{Q}}\mu S+\Delta$,
where $\Delta$ is an effective $\mathbb{Q}$-divisor, and $\mu$ is a positive rational number such that $\mu>\frac{4}{3}$.
Then $S$ is one of the surfaces $\widetilde{Q}$, $H_1$, $H_2$, $H_3$, $H_4$.
\end{claim}

\begin{proof}
This follows from the fact that the cone $\mathrm{Eff}(X)$ is generated by $E$ and $\widetilde{Q}$.
\end{proof}

Suppose $X$ is not K-stable. Since $\mathrm{Aut}(X)$ is finite, the~threefold $X$ is not~K-polystable.
Then, by \cite[Corollary~4.14]{Zhuang}, there is a~$G$-invariant prime divisor $F$ over $X$ with $\beta(F)\leqslant 0$,
see \cite[\S~1.2]{ACCFKMGSSV} for the precise definition of $\beta(F)$.
Let us seek for a~contradiction.

Let $Z$ be the center of $F$ on $X$.
Then $Z$ is not a~surface by \cite[Theorem~3.7.1]{ACCFKMGSSV},
so that $Z$ is either a~$G$-invariant irreducible curve or a~$G$-fixed point.
In the latter case, the point $\pi(Z)$ must be one of the $G$-fixed points $O_1$, $O_2$, $O_3$, $O_4$,
so that the~point $Z$ is not contained in $\widetilde{Q}\cup E$.
Let us use Abban--Zhuang theory \cite{AbbanZhuang} to show that $Z$ does not lie on~$\widetilde{Q}\cup E$ in the former case.

\begin{lemma}
The center $Z$ cannot be contained in $\widetilde{Q}\cup E$.
\end{lemma}

\begin{proof}
We suppose that $Z\subset \widetilde{Q}\cup E$. Then $Z$ is an irreducible $G$-invariant curve,
because neither $\widetilde{Q}$ nor $E$ contains $G$-fixed points.
Let us use~notations introduced in~\mbox{\cite[\S~1.7]{ACCFKMGSSV}}.
Namely, we fix $u\in\mathbb{R}_{\geqslant 0}$.
Then $$-K_X-u\widetilde{Q}\sim_{\mathbb{R}} (4-2u)H+(u-1)E\sim_{\mathbb{R}} (1-u)\widetilde{Q}+2H,$$
so that $-K_X-u\widetilde{Q}$ is nef for $0\leqslant u\leqslant 1$, and  not pseudo-effective for $u>2$.
Thus, we have
$$
P\big(-K_X-u\widetilde{Q}\big)=
\left\{\aligned
&-K_X-u\widetilde{Q} \ & \text{if $0\leqslant u\leqslant 1$}, \\
&(4-2u)H\ &  \text{if $1\leqslant u\leqslant 2$},
\endaligned
\right.
$$
and
$$
N\big(-K_X-u\widetilde{Q}\big)=
\left\{\aligned
&0\ & \text{if $0\leqslant u\leqslant 1$}, \\
&(u-1)E\ & \text{if $1\leqslant u\leqslant 2$}.
\endaligned
\right.
$$
If $Z\subset\widetilde{Q}$, then \cite[Corollary~1.7.26]{ACCFKMGSSV} gives
$$
1\geqslant\frac{A_X(F)}{S_X(F)}\geqslant\min\Bigg\{\frac{1}{S_X(\widetilde{Q})},\frac{1}{S\big(W^{\widetilde{Q}}_{\bullet,\bullet};Z\big)}\Bigg\},
$$
where
$$
S_X\big(\widetilde{Q}\big)=\frac{1}{(-K_X)^3}\varint_0^2\mathrm{vol}\big(-K_X-u\widetilde{Q}\big)du=\frac{1}{(-K_X)^3}\varint_0^2\Big(P\big(-K_X-u\widetilde{Q}\big)\Big)^{3}du
$$
and
\begin{multline*}
S\big(W^{\widetilde{Q}}_{\bullet,\bullet};Z\big)=
\frac{3}{(-K_X)^3}\left\{\varint_0^2\Big(P\big(-K_X-u\widetilde{Q}\big)^{2}\cdot\widetilde{Q}\Big)\cdot\mathrm{ord}_{Z}\Big(N\big(-K_X-u\widetilde{Q}\big)\big\vert_{\widetilde{Q}}\Big)du+\right.\\
\left.+\varint_0^2\varint_0^\infty \mathrm{vol}\Big(P\big(-K_X-u\widetilde{Q}\big)\big\vert_{\widetilde{Q}}-vZ\Big)dvdu\right\}.%
\end{multline*}
Therefore, we conclude that $S(W_{\bullet,\bullet}^{\widetilde{Q}};Z)\geqslant 1$, because $S_X(\widetilde{Q})<1$, see \cite[Theorem~3.7.1]{ACCFKMGSSV}.
Similarly, if $Z\subset E$, then we get $S(W_{\bullet,\bullet}^{E};Z)\geqslant 1$.

Fix an~isomorphism $\widetilde{Q}\cong Q\cong\mathbb{P}^1\times\mathbb{P}^1$ such that $E\vert_{\widetilde{Q}}$ is a divisor in $\widetilde{Q}$ of degree $(1,3)$.
For $(a,b)\in\mathbb{R}^2$, let $\mathcal{O}_{\widetilde{Q}}(a,b)$ be the class of a divisor of degree $(a,b)$  in $\mathrm{Pic}(\widetilde{Q})\otimes\mathbb{R}$. Then
$$
P(-K_X-u\widetilde{Q})\vert_{\widetilde{Q}}\sim_{\mathbb{R}}
\left\{\aligned
&\mathcal{O}_{\widetilde{Q}}(3-u,u+1)\ & \text{if $0\leqslant u\leqslant 1$}, \\
&\mathcal{O}_{\widetilde{Q}}(4-2u,4-2u)\ & \text{if $1\leqslant u\leqslant 2$}.
\endaligned
\right.
$$
Therefore, if $Z=E\cap\widetilde{Q}$, then
\[\aligned
S\big(W_{\bullet,\bullet}^{\widetilde{Q}};Z\big)&=\frac{1}{10}\left\{\varint_{1}^{2}2(4-2u)^2(u-1)du+
\varint_{0}^{1}\varint_{0}^{\infty}\mathrm{vol}\Big(\mathcal{O}_{\widetilde{Q}}(3-u-v,u+1-3v)\Big)dvdu\right.\\
&\pushright{+\left.\varint_{1}^{2}\varint_{0}^{\infty}\mathrm{vol}\Big(\mathcal{O}_{\widetilde{Q}}(4-2u-v,4-2u-3v)\Big)dvdu\right\}}\\
&=\frac{2}{30}+\frac{1}{10}\left\{\varint_{0}^{1}\varint_{0}^{\frac{u+1}{3}}2(u+1-3v)(3-u-v)dvdu\right.\\
&\pushright{\left.+\varint_{1}^{2}\varint_{0}^{\frac{4-2u}{3}}2(4-2u-3v)(4-2u-v)dvdu\right\}}\\
&=\frac{161}{540}.\\
\endaligned
\]
To estimate $S(W_{\bullet,\bullet}^{\widetilde{Q}};Z)$ in the case when $Z\subset\widetilde{Q}$ and $Z\ne E\cap\widetilde{Q}$,
observe that $|Z-\Delta|\ne\varnothing$, where $\Delta$ is the~diagonal curve in $\widetilde{Q}$.
Indeed, this follows from the fact that $\widetilde{Q}$ contains neither $G$-invariant curves of degree $(0,1)$
nor $G$-invariant curves of degree $(1,0)$, which in turns easily follows from the fact that the~curve $C_4\cong\mathbb{P}^1$ does not have $G$-fixed points.
Thus, if  $Z\subset\widetilde{Q}$ and $Z\ne E\cap\widetilde{Q}$, then
\[\aligned
S\big(W_{\bullet,\bullet}^{\widetilde{Q}};Z\big)&\leqslant
\frac{1}{10}\varint_0^2\varint_0^\infty \mathrm{vol}\Big(P\big(-K_X-u\widetilde{Q}\big)\big\vert_{\widetilde{Q}}-v\Delta\Big)dvdu\\
&=\frac{1}{10}\left\{\varint_{0}^{1}\varint_{0}^{\infty}\mathrm{vol}\Big(\mathcal{O}_{\widetilde{Q}}(3-u-v,u+1-v)\Big)dvdu+\right.\\
&\pushright{\left.+\varint_{1}^{2}\varint_{0}^{\infty}\mathrm{vol}\Big(\mathcal{O}_{\widetilde{Q}}(4-2u-v,4-2u-v)\Big)dvdu\right\}}\\
&=\frac{1}{10}\left\{\varint_{0}^{1}\varint_{0}^{u+1}2(u+1-v)(3-u-v)dvdu+\varint_{1}^{2}\varint_{0}^{4-2u}2(4-2u-v)^2dvdu\right\}\\
&=\frac{17}{30}.
\endaligned\]
Therefore,  $Z\not\subset\widetilde{Q}$, and hence $Z\subset E$ and $Z\ne\widetilde{Q}\cap E$.

One has $E\cong\mathbb{F}_n$ for some integer $n\geqslant 0$. It follows from the argument as in the proof of \cite[Lemma~4.4.16]{ACCFKMGSSV}
that $n$ is either $0$ or $2$.
Indeed, let $\mathbf{s}$ be the section of the~projection $E\to C_4$ such that $\mathbf{s}^2=-n$,
and let $\mathbf{l}$ be its~fiber. Then $-E\big\vert_{E}\sim \mathbf{s}+k\mathbf{l}$
for some integer~$k$. But
$$
-n+2k=E^3=-\mathrm{c}_1(\mathcal{N}_{C_4/\mathbb{P}^3})=-14,
$$
so that $k=\frac{n-14}{2}$. Then
$$
\widetilde{Q}\big\vert_{E}\sim \big(2H-E\big)\big\vert_{E}\sim\mathbf{s}+(k+8)\mathbf{l}=\mathbf{s}+\frac{n+2}{2}\mathbf{l},
$$
which implies that $\widetilde{Q}\vert_{E}\not\sim\mathbf{s}$.
Moreover, we know that $\widetilde{Q}\vert_{E}$ is a~smooth irreducible curve, since the~quadric surface $Q$ is smooth.
Thus, since $\widetilde{Q}\big\vert_{E}\ne\mathbf{s}$, we have
$$
0\leqslant \widetilde{Q}\big\vert_{E}\cdot\mathbf{s}=\Big(\mathbf{s}+\frac{n+2}{2}\mathbf{l}\Big)\cdot\mathbf{s}=-n+\frac{n+2}{2}=\frac{2-n}{2}
$$
so that $n=0$ or $n=2$. Now, let us show that $S(W_{\bullet,\bullet}^{E};Z)<1$ in both cases.

For $u\geqslant 0$,
$$-K_X-uE\sim 2\widetilde{Q}+(1-u)E,$$
so that $-K_X-uE$ is pseudo-effective if and only if $u\leqslant 1$,
and it is nef if and only if $u\leqslant\frac{1}{3}$.
Furthermore, if $\frac{1}{3}\leqslant u\leqslant 1$, then
$$
P(-K_X-uE)=(2-2u)(3H-E)
$$
and $N(-K_X-uE)=(3u-1)\widetilde{Q}$. Thus, if $n=0$, we have
$$
P(-K_X-uE)\big\vert_{E}=
\left\{\aligned
&(1+u)\mathbf{s}+(9-7u)\mathbf{l}\ & \text{if $0\leqslant u\leqslant \frac{1}{3}$}, \\
&(2-2u)\mathbf{s}+(10-10u)\mathbf{l}\ & \text{if $\frac{1}{3}\leqslant u\leqslant 1$}.
\endaligned
\right.
$$
Similarly, if $n=2$, then
$$
P(-K_X-uE)\big\vert_{E}=
\left\{\aligned
&(1+u)\mathbf{s}+(10-6u)\mathbf{l}\ & \text{if $0\leqslant u\leqslant \frac{1}{3}$}, \\
&(2-2u)\mathbf{s}+(12-12u)\mathbf{l}\ & \text{if $\frac{1}{3}\leqslant u\leqslant 1$}.
\endaligned
\right.
$$
Recall that $Z\ne \widetilde{Q}\cap E$. Moreover, we have $Z\not\sim\mathbf{l}$, since $\pi(Z)$ is not one of the $G$-fixed points $O_1$, $O_2$, $O_3$, $O_4$.
Thus, using \cite[Corollary~1.7.26]{ACCFKMGSSV}, we get
$$
S(W_{\bullet,\bullet}^{E};Z)=\frac{1}{10}\varint_{0}^{1}\varint_{0}^{\infty}\mathrm{vol}\Big(P(u)\big\vert_{E}-vZ\Big)dvdu\leqslant\frac{1}{10}\varint_{0}^{1}\varint_{0}^{\infty}\mathrm{vol}\Big(P(u)\big\vert_{E}-v\mathbf{s}\Big)dvdu,
$$
because the~divisor $|Z-\mathbf{s}|\ne\varnothing$.

Consequently, if $n=0$, then
\[\aligned
\phantom{s}&
S(W_{\bullet,\bullet}^{E};Z) \leqslant\\
&\phantom{=}\frac{1}{10}\left\{\varint_{0}^{\frac{1}{3}}\varint_{0}^{\infty}\mathrm{vol}\Big((1+u)\mathbf{s}+(9-7u)\mathbf{l}-v\mathbf{s}\Big)dvdu+\right.\\
&\pushright{\left.+\varint_{\frac{1}{3}}^{1}\varint_{0}^{\infty}\mathrm{vol}\Big((2-2u)\mathbf{s}+(10-10u)\mathbf{l}-v\mathbf{s}\Big)dvdu\right\}}\\
&=\frac{1}{10}\left\{\varint_{0}^{\frac{1}{3}}\varint_{0}^{1+u}2(1+u-v)(9-7u)dvdu+\varint_{\frac{1}{3}}^{1}\varint_{0}^{2-2u}2(2-2u-v)(10-10u)dvdu\right\}\\
&=\frac{1783}{3240}.
\endaligned
\]
Similarly, if $n=2$, then
\[\aligned
\phantom{s}&
S(W_{\bullet,\bullet}^{E};Z)\leqslant \\
&\phantom{=}\frac{1}{10}\left\{\varint_{0}^{\frac{1}{3}}\varint_{0}^{\infty}\mathrm{vol}\Big((1+u)\mathbf{s}+(10-6u)\mathbf{l}-v\mathbf{s}\Big)dvdu+\right.\\
&\pushright{\left.+\varint_{\frac{1}{3}}^{1}\varint_{0}^{\infty}\mathrm{vol}\Big((2-2u)\mathbf{s}+(12-12u)\mathbf{l}-v\mathbf{s}\Big)dvdu\right\}}\\
&=\frac{1}{10}\left\{\varint_{0}^{\frac{1}{3}}\varint_{0}^{1+u}2(1+u-v)(9+v-7u)dvdu+
\varint_{\frac{1}{3}}^{1}\varint_{0}^{2-2u}2(2-2u-v)(10+v-10u)dvdu\right\}\\
&=\frac{157}{270}.\\
\endaligned
\]
In both cases, we have $S(W_{\bullet,\bullet}^{E};Z)<1$, which is a~contradiction.
\end{proof}

Now, we prove our main technical result using Abban--Zhuang theory, see also \cite[\S~1.7]{ACCFKMGSSV}.

\begin{proposition}
The center $Z$ is not contained in $H_1\cup H_2\cup H_3\cup H_4$.
\end{proposition}

\begin{proof}
We first suppose that $Z\subset H_1\cup H_2\cup H_3$.
Without loss of generality, we may assume that~$Z\subset H_1$. Then~$\pi(Z)\subset\Pi_1$.
Therefore, we see that one of the following two subcases are possible:
\begin{itemize}
\item either $\pi(Z)$ is one of the $G$-fixed points $O_2$, $O_3$, $O_4$,
\item or $Z$ is a~$G$-invariant irreducible curve in $H_1$.
\end{itemize}
We will deal with these subcases separately. In both subcases, we let $S=H_1$ for simplicity.
Recall that $S$ is a~smooth del Pezzo surface of degree $5$,
the surface $S$ is $G$-invariant, and the~action of the~group $G$ on the~surface $S$ is faithful.
Note also that $Z\not\subset\widetilde{Q}$ by Lemma.

Let us use notations introduced in \cite[\S~1.7]{ACCFKMGSSV}. Take $u\in\mathbb{R}_{\geqslant 0}$. Then
$$
-K_X-uS\sim_{\mathbb{R}} (4-u)H-E\sim_{\mathbb{R}} \widetilde{Q}+(2-u)H\sim_{\mathbb{R}} (u-1)\widetilde{Q}+(2-u)\big(3H-E\big).
$$
Let $P(u)=P(-K_X-uS)$ and $N(u)=N(-K_X-uS)$. Then
\[
P\big(u\big)=
\left\{\aligned
&-K_X-uS\ & \text{if $0\leqslant u\leqslant 1$}, \\
&(2-u)\big(3H-E\big)\ & \text{if $1\leqslant u\leqslant 2$},
\endaligned
\right.
\]
and
\[
N\big(u\big)=
\left\{\aligned
&0\ & \text{if $0\leqslant u\leqslant 1$}, \\
&(u-1)\widetilde{Q}\  & \text{if $1\leqslant u\leqslant 2$}.
\endaligned
\right.
\]
Note that $S_X(S)<1$, see \cite[Theorem~3.7.1]{ACCFKMGSSV}.
In fact, one can compute $S_X(S)=\frac{17}{30}$.

Let $\varphi\colon S\to \Pi_1$ be birational morphism induced by $\pi$.
Then $\varphi$ is a~$G$-equivariant blow up of the~four intersection points $\Pi_1\cap C_4$.
Let $\ell$ be the~proper transform on $S$ of a~general line in $\Pi_1$, and let $e_1$, $e_2$, $e_3$, $e_4$ be $\varphi$-exceptional curves,
and let $\ell_{ij}$ be the proper transform on the~surface $S$ of the line in $\Pi_1$ that passes through $\varphi(e_i)$ and $\varphi(e_j)$, where $1\leqslant i<j\leqslant 4$.
Then the cone $\overline{\mathrm{NE}}(S)$ is generated by the curves  $e_1$, $e_2$, $e_3$, $e_4$,
$\ell_{12}$, $\ell_{13}$, $\ell_{14}$, $\ell_{23}$, $\ell_{24}$, $\ell_{34}$.
Recall also that
$$
\Pi_1\cap C_4=\big(L_{24}\cap C_4\big)\cup\big(L_{34}\cap C_4\big).
$$
Therefore, we may assume that $L_{24}\cap C_4=\varphi(e_1)\cup\varphi(e_2)$ and $L_{34}\cap C_4=\varphi(e_3)\cup\varphi(e_4)$,
so that we have $\varphi(\ell_{12})=L_{24}$ and $\varphi(\ell_{34})=L_{34}$.

Observe that, the~group $\mathrm{Pic}^G(S)$ is generated by the~divisor classes $\ell$, $e_1+e_2$, $e_3+e_4$,
because both $L_{24}\cap C_4$ and $L_{34}\cap C_4$ are $G$-orbits of length $2$.
Therefore, if $Z$ is a curve, then $\varphi(Z)$ is a curve of degree $d\geqslant 1$, so that
$Z\sim d\ell-m_{12}(e_1+e_2)-m_{34}(e_3+e_4)$ for some non-negative integers $m_{12}$ and $m_{34}$, which gives
\[\aligned
Z&\sim (d-2m_{12})\ell+m_{12}(2\ell-e_1-e_2-e_3-e_4)+(m_{12}-m_{34})(e_3+e_4)\\
&\sim (d-2m_{12})(\ell_{12}+e_1+e_2)+m_{12}(\ell_{12}+\ell_{34})+(m_{12}-m_{34})(e_3+e_4)
\endaligned
\]
and
\[\aligned
Z&\sim (d-2m_{34})\ell+m_{34}(2\ell-e_1-e_2-e_3-e_4)+(m_{34}-m_{12})(e_1+e_2)\\
&\sim (d-2m_{34})(\ell_{34}+e_3+e_4)+m_{34}(\ell_{12}+\ell_{34})+(m_{34}-m_{12})(e_1+e_2).
\endaligned
\]
Moreover, if $Z\ne\ell_{12}$ and $Z\ne\ell_{34}$, then
$d-2m_{12}=Z\cdot\ell_{12}\geqslant 0$ and $d-2m_{34}=Z\cdot\ell_{34}\geqslant 0$.
Hence, if $Z$ is a curve, then $|Z-\ell_{12}|\ne\varnothing$ or $|Z-\ell_{34}|\ne\varnothing$.

On the other hand, if $Z$ is a curve, then \cite[Corollary~1.7.26]{ACCFKMGSSV} gives
$$
1\geqslant\frac{A_X(F)}{S_X(F)}\geqslant\min\Bigg\{\frac{1}{S_X(S)},\frac{1}{S\big(W^S_{\bullet,\bullet};Z\big)}\Bigg\}=\min\Bigg\{\frac{30}{17},\frac{1}{S\big(W^S_{\bullet,\bullet};Z\big)}\Bigg\},
$$
where
$$
S\big(W^S_{\bullet,\bullet};Z\big)=\frac{3}{(-K_X)^3}\varint_0^2\varint_0^\infty \mathrm{vol}\big(P(u)\big\vert_{S}-vZ\big)dvdu,
$$
because $Z\not\subset\widetilde{Q}$.
Moreover, if $S(W^S_{\bullet,\bullet};Z)=1$, then \cite[Corollary~1.7.26]{ACCFKMGSSV} gives
$$
1\geqslant\frac{A_X(E)}{S_X(E)}=\frac{1}{S_X(S)}=\frac{30}{17},
$$
which is absurd. Thus, if $Z$ is a curve, then $S(W^S_{\bullet,\bullet};Z)>1$, which gives
\[\aligned
1<S\big(W^S_{\bullet,\bullet}&;Z\big)=\frac{1}{10}\varint_0^2\varint_0^\infty \mathrm{vol}\big(P(u)\big\vert_{S}-vZ\big)dvdu\\
&\leqslant\max\Bigg\{
\frac{1}{10}\varint_0^2\varint_0^\infty \mathrm{vol}\big(P(u)\big\vert_{S}-v\ell_{12}\big)dvdu,
\frac{1}{10}\varint_0^2\varint_0^\infty \mathrm{vol}\big(P(u)\big\vert_{S}-v\ell_{34}\big)dvdu\Bigg\},\\
\endaligned
\]
because $|Z-\ell_{12}|\ne\varnothing$ or $|Z-\ell_{34}|\ne\varnothing$. Note also that
$$
S\big(W_{\bullet,\bullet}^S;\ell_{12}\big)=\frac{1}{10}\varint_0^2\varint_0^\infty \mathrm{vol}\big(P(u)\big\vert_{S}-v\ell_{12}\big)dvdu=\frac{1}{10}\varint_0^2\varint_0^\infty \mathrm{vol}\big(P(u)\big\vert_{S}-v\ell_{34}\big)dvdu.
$$
Hence, if $Z$ is a curve, then the second statement in  \cite[Corollary~1.7.26]{ACCFKMGSSV} gives
$$
1<S\big(W^S_{\bullet,\bullet};Z\big)\leqslant S\big(W_{\bullet,\bullet}^S;\ell_{12}\big)=\frac{1}{10}\varint_0^2\varint_0^\infty \mathrm{vol}\big(P(u)\big\vert_{S}-v\ell_{12}\big)dvdu.
$$
Let us compute $S(W_{\bullet,\bullet}^S;\ell_{12})$.
For  $0\leqslant u\leqslant 1$ and $v\geqslant 0$, we have
$$
P(u)\big\vert_{S}-v\ell_{12}=\big(-K_X-uS\big)\big\vert_{S}-v\ell_{12}\sim_{\mathbb{R}}(4-u-v)\ell-(1-v)\big(e_1+e_2\big)-e_3-e_4.
$$
Therefore, if $0\leqslant v\leqslant 1$, then this divisor is nef, and its volume is $u^2+2uv-v^2-8u-4v+12$.
Similarly, if $1\leqslant v\leqslant 2-u$, then its Zariski decomposition is
$$
P(u)\big\vert_{S}-v\ell_{12}\sim_{\mathbb{R}}
\underbrace{(4-u-v)\ell-e_3-e_4}_{\text{positive part}}+\underbrace{(v-1)\big(e_1+e_2\big)}_{\text{negative part}},
$$
so that its volume is $u^2+2uv+v^2-8u-8v+14$.
Likewise, if $2-u\leqslant v\leqslant 3-u$, then the Zariski decomposition of the divisor $P(u)\vert_{S}-v\ell_{12}$ is
$$
P(u)\big\vert_{S}-v\ell_{12}\sim_{\mathbb{R}}
\underbrace{(3-u-v)(2\ell-e_3-e_4)}_{\text{positive part}}+\underbrace{(v-1)\big(e_1+e_2\big)+(v-2+u)\ell_{34}}_{\text{negative part}},
$$
so that its volume is $2(3-u-v)^2$. If $v>3-u$, then $P(u)\vert_{S}-v\ell_{12}$ is not pseudo-effective, so that the volume of this divisor is zero.
Thus, we have
\[
\aligned
\frac{1}{10}\varint_{0}^{1}\varint_{0}^{\infty}\mathrm{vol}&\Big(P(u)\big\vert_{S}-v\ell_{12}\Big)dvdu\\
&=\frac{1}{10}\varint_{0}^{1}\varint_{0}^{3-u}\mathrm{vol}\Big(P(u)\big\vert_{S}-v\ell_{12}\Big)dvdu\\
&=\frac{1}{10}\left\{\varint_{0}^{1}\varint_{0}^{1}\big(u^2+2uv-v^2-8u-4v+12\big)dvdu+\right.\\
&\pushright{\phantom{space}\left.+\varint_{0}^{1}\varint_{1}^{2-u}\big(u^2+2uv+v^2-8u-8v+14\big)dvdu+\varint_{0}^{1}\varint_{2-u}^{3-u}2(3-u-v)^2dvdu\right\}}\\
&=\frac{107}{120}.
\endaligned
\]
Similarly, if $1\leqslant u\leqslant 2$, then
$$
P(u)\big\vert_{S}-v\ell_{12}\sim_{\mathbb{R}}(6-3u-v)\ell+(v+u-2)(e_1+e_2)+(u-2)(e_3+e_4).
$$
If $0\leqslant v\leqslant 2-u$, this divisor is nef, and its volume is $5u^2+2uv-v^2-20u-4v+20$.
Likewise, if $2-u\leqslant v\leqslant 4-2u$, then its Zariski decomposition is
$$
P(u)\big\vert_{S}-v\ell_{12}\sim_{\mathbb{R}}\underbrace{(4-2u-v)(2\ell-e_3-e_4)}_{\text{positive part}}+
\underbrace{(v-2+u)\big(e_1+e_2+\ell_{34}\big)}_{\text{negative part}},
$$
and its volume is $2(4-2u-v)^2$. If $v>4-2u$, this divisor is not pseudo-effective, so that
\[
\aligned
\frac{1}{10}\varint_{1}^{2}&\varint_{0}^{\infty}\mathrm{vol}\Big(P(u)\big\vert_{S}-v\ell_{12}\Big)dvdu\\
&=\frac{1}{10}\varint_{1}^{2}\varint_{0}^{4-2u}\mathrm{vol}\Big(P(u)\big\vert_{S}-v\ell_{12}\Big)dvdu\\
&=\frac{1}{10}\left\{\varint_{1}^{2}\varint_{0}^{2-u}\big(5u^2+2uv-v^2-20u-4v+20\big)dvdu+\varint_{1}^{2}\varint_{2-u}^{4-2u}2(4-2u-v)^2dvdu\right\}\\
&=\frac{13}{120}.
\endaligned
\]
Therefore, we see that
\[\aligned
 S\big(W_{\bullet,\bullet}^S;\ell_{12}\big)&=\frac{1}{10}\varint_0^2\varint_0^\infty \mathrm{vol}\big(P(u)\big\vert_{S}-v\ell_{12}\big)dvdu\\
&=\frac{1}{10}\varint_{0}^{1}\varint_{0}^{\infty}\mathrm{vol}\Big(P(u)\big\vert_{S}-v\ell_{12}\Big)dvdu+\frac{1}{10}\varint_{1}^{2}\varint_{0}^{\infty}\mathrm{vol}\Big(P(u)\big\vert_{S}-v\ell_{12}\Big)dvdu\\
&=\frac{107}{120}+\frac{13}{120}=1,
\endaligned
\]
which implies, in particular, that $Z$ is not a curve.

Hence, we see that $\pi(Z)$ is one of the points $O_2$, $O_3$, $O_4$.
Without loss of generality, we may assume that either $\pi(Z)=O_2$ or $\pi(Z)=O_4$, so that $Z\in\ell_{12}$ in both subcases.
Now, using \cite[Theorem~1.7.30]{ACCFKMGSSV}, we see that
$$
1\geqslant\frac{A_X(F)}{S_X(F)}\geqslant \min\left\{\frac{1}{S(W_{\bullet, \bullet,\bullet}^{S,\ell_{12}};Z)},
\frac{1}{S(W_{\bullet,\bullet}^S;\ell_{12})},\frac{1}{S_X(S)}\right\}=
\min\left\{\frac{1}{S(W_{\bullet, \bullet,\bullet}^{S,\ell_{12}};Z)},1\right\},
$$
where $S(W_{\bullet, \bullet,\bullet}^{S,\ell_{12}};Z)$ is defined in \cite[\S~1.7]{ACCFKMGSSV}.
In fact, \cite[Theorem~1.7.30]{ACCFKMGSSV} implies the~strict inequality $S(W_{\bullet, \bullet,\bullet}^{S,\ell_{12}};Z)<1$, because $S_X(S)<1$.
Let us compute $S(W_{\bullet, \bullet,\bullet}^{S,\ell_{12}};Z)$.

For $0\leqslant u\leqslant 2$ and $v\geqslant 0$, let $P(u,v)$ be the~positive part of the~Zariski decomposition of the~divisor $P(u)\vert_{S}-v\ell_{12}$, and let $N(u,v)$ be its~negative part.

If $0\leqslant u\leqslant 1$, then
$$
P(u,v)=\left\{\aligned
&(4-u-v)\ell-(1-v)\big(e_1+e_2\big)-e_3-e_4\ & \text{if $0\leqslant v\leqslant 1$}, \\
&(4-u-v)\ell-e_3-e_4\ & \text{if $1\leqslant v\leqslant 2-u$}, \\
&(3-u-v)(2\ell-e_3-e_4)\  & \text{if $2-u\leqslant v\leqslant 3-u$},
\endaligned
\right.
$$
and
$$
N(u,v)=\left\{\aligned
&0\ & \text{if $0\leqslant v\leqslant 1$}, \\
&(v-1)\big(e_1+e_2\big)\ & \text{if $1\leqslant v\leqslant 2-u$}, \\
&(v-1)\big(e_1+e_2\big)+(v-2+u)\ell_{34}\  & \text{if $2-u\leqslant v\leqslant 3-u$}.
\endaligned
\right.
$$
Similarly, if $1\leqslant u\leqslant 2$, then
$$
P(u,v)=\left\{\aligned
&(6-3u-v)\ell+(v+u-2)(e_1+e_2)+(u-2)(e_3+e_4)  \ \  \text{ if $0\leqslant v\leqslant 2-u$}, \\
&(4-2u-v)\big(2\ell-e_3-e_4\big)   \text{ \hspace{43.5mm} if $2-u\leqslant v\leqslant 4-2u$},
\endaligned
\right.
$$
and
$$
N(u,v)=\left\{\aligned
&0\ & \text{if $0\leqslant v\leqslant 2-u$}, \\
&(v-2+u)\big(e_1+e_2+\ell_{34}\big)\ & \text{if $2-u\leqslant v\leqslant 4-2u$}.
\endaligned
\right.
$$
Recall~from \cite[Theorem~1.7.30]{ACCFKMGSSV} that
$$
S\big(W_{\bullet,\bullet}^{S,\ell_{12}};Z\big)=F_{Z}\big(W_{\bullet,\bullet}^{S,\ell_{12}}\big)+\frac{3}{(-K_X)^3}\varint_{0}^{2}\varint_{0}^{\infty}\big(P(u,v)\cdot\ell_{12}\big)^2dvdu
$$
for
$$
F_Z\big(W_{\bullet,\bullet}^{S,\ell_{12}}\big)=\frac{6}{(-K_X)^3}\varint_{0}^{2}\varint_{0}^{\infty}\big(P(u,v)\cdot\ell_{12}\big)
\mathrm{ord}_Z\Big(N^\prime_{S}(u)\big\vert_{\ell_{12}}+N(u,v)\big\vert_{\ell_{12}}\Big)dvdu,
$$
where $N^\prime_S(u)$ is the~part of the~divisor $N(u)\vert_{S}$ whose support does not contain~$\ell_{12}$,
so that $N^\prime_S(u)=N(u)\vert_{S}$ in our case,
which implies that $\mathrm{ord}_{Z}(N^\prime_{S}(u)\vert_{\ell_{12}})=0$ for $0\leqslant u\leqslant 2$, because $Z\not\in\widetilde{Q}$.
Thus, if $\pi(Z)=O_2$, then $Z\not\in \ell_{34}\cup e_1\cup e_2$, which gives $F_Z(W_{\bullet,\bullet}^{S,\ell_{12}})=0$.
On the other hand, if $\pi(Z)=O_4$, then $Z=\ell_{12}\cap\ell_{34}$ and $Z\not\in e_1\cup e_2$, so that
\[
\aligned
F_Z\big(W_{\bullet,\bullet}^{S,\ell_{12}}\big) &=
\frac{1}{5}\varint_{0}^{2}\varint_{0}^{\infty}\big(P(u,v)\cdot\ell_{12}\big)\mathrm{ord}_Z\Big(N(u,v)\big\vert_{\ell_{12}}\Big)dvdu\\
&=\frac{1}{5}\left\{\varint_{0}^{1}\varint_{2-u}^{3-u}(6-2u-2v+6)(v-2+u)dvdu+\right.\\
& \pushright{\hspace{50mm}\left.+\varint_{1}^{2}\varint_{2-u}^{4-2u}(8-4u-2v+8)(v-2+u)dvdu\right\}}\\
&=\frac{1}{12}.
\endaligned
\]
Therefore, we see that
\[\aligned
  S\big(&W_{\bullet,\bullet}^{S,\ell_{12}};Z\big)\leqslant \frac{1}{12}+\frac{1}{10}\varint_{0}^{2}\varint_{0}^{\infty}\big(P(u,v)\cdot\ell_{12}\big)^2dvdu\\
&=\frac{1}{12}+\frac{1}{10}\left\{\varint_{0}^{1}\varint_{0}^{1}\big(2-u+v\big)^2dvdu+\varint_{0}^{1}\varint_{1}^{2-u}\big(4-u-v\big)^2dvdu+\right.\\
&\left.+\varint_{0}^{1}\varint_{2-u}^{3-u}\big(6-2u-2v\big)^2dvdu+\varint_{1}^{2}\varint_{0}^{2-u}\big(2-u+v\big)^2dvdu+\varint_{1}^{2}\varint_{2-u}^{4-2u}\big(8-4u-2v\big)^2dvdu\right\}\\
&=1.
\endaligned
\]
However, as we already mentioned, one has $S(W_{\bullet,\bullet}^{S,\ell_{12}};Z)<1$ by  \cite[Theorem~1.7.30]{ACCFKMGSSV}.
The~obtained contradiction concludes that  $Z\subset H_4$.

Since  $Z\not\subset H_1\cup H_2\cup H_3$, the center~$Z$ must be a $G$-invariant curve on $H_4$. 
Moreover,~$\pi(Z)$ cannot be one of the lines determined by the points $O_1$, $O_2$, $O_3$ on $\Pi_4$. This implies that~$\pi(Z)$ is a curve of degree $d\geqslant 2$ on $\Pi_4$.

 We keep the same notations as in the beginning of the proof, i.e., put $S=H_4$ and
let $\varphi\colon S\to \Pi_1$ be birational morphism induced by $\pi$.
As before,  $\varphi$ is a~$G$-equivariant blow up of the~four intersection points $\Pi_4\cap C_4$ which consist of a $G$- orbit of length $4$. We also denote by $\ell$  the~proper transform on $S$ of a~general line in $\Pi_4$ and by  $e_1$, $e_2$, $e_3$, $e_4$  the four $\varphi$-exceptional curves. In addition, denote by $\mathscr{C}$ the proper transform of a general conic passing through the four points  $\Pi_4\cap C_4$.

Since the~group $\mathrm{Pic}^G(S)$ is generated by the~divisor classes $\ell$, $e_1+e_2+e_3+e_4$, we have
$$
Z\sim d\ell-m(e_1+e_2+e_3+e_4).
$$
where $m$ is a non-negative integer. By taking intersection with the proper transforms of the lines on $\Pi_4$ passing through $\varphi(e_i)$, $\varphi(e_j)$,
we obtain $d\geqslant 2m$. Since $d\geqslant 2$, this implies that $|Z-\mathscr{C}|\ne\varnothing$. Note that $\mathscr{C}\not\subset \widetilde{Q}$.
By the same argument as before, we obtain
\[\begin{split}
1<S\big(W^S_{\bullet,\bullet};Z\big)&=\frac{1}{10}\varint_0^2\varint_0^\infty \mathrm{vol}\big(P(u)\big\vert_{S}-vZ\big)dvdu\\
&\leqslant \frac{1}{10}\varint_0^2\varint_0^\infty \mathrm{vol}\big(P(u)\big\vert_{S}-v\mathscr{C}\big)dvdu=S\big(W^S_{\bullet,\bullet};\mathscr{C}\big),\\
\end{split}\]
where $P(u)$ is the positive part of $-K_X-uS$ as before.
Let us compute $S(W^S_{\bullet,\bullet};\mathscr{C})$.

Similar to the notations used earlier in the proof, we denote by $P(u,v)$ the~positive part of the~Zariski decomposition of the~divisor $P(u)\vert_{S}-v\mathscr{C}$ for $0\leqslant u\leqslant 2$ and $v\geqslant 0$, and we denote by $N(u,v)$ its~negative part.
If $0\leqslant u\leqslant 1$, then
$$
P(u,v)=\left\{\aligned
&(4-u-2v)\ell-(1-v)\big(e_1+e_2+e_3+e_4\big)\ & \text{if $0\leqslant v\leqslant 1$}, \\
&(4-u-2v)\ell\  & \text{if $1\leqslant v\leqslant \frac{4-u}{2}$},
\endaligned
\right.
$$
and
$$
N(u,v)=\left\{\aligned
&0\ & \text{if $0\leqslant v\leqslant 1$}, \\
&(v-1)\big(e_1+e_2+e_3+e_4\big)\  & \text{if $1\leqslant v\leqslant \frac{4-u}{2}$}.
\endaligned
\right.
$$
Similarly, if $1\leqslant u\leqslant 2$, then
$$
P(u,v)=\left\{\aligned
&(6-3u-2v)\ell+(v+u-2)(e_1+e_2+e_3+e_4)  \ \  \text{ if $0\leqslant v\leqslant 2-u$}, \\
&(6-3u-2v)\ell   \text{ \hspace{43.5mm} if $2-u\leqslant v\leqslant \frac{6-3u}{2}$},
\endaligned
\right.
$$
and
$$
N(u,v)=\left\{\aligned
&0\ & \text{if $0\leqslant v\leqslant 2-u$}, \\
&(v+u-2)\big(e_1+e_2+e_3+e_4\big)\ & \text{if $2-u\leqslant v\leqslant \frac{6-3u}{2}$}.
\endaligned
\right.
$$
This gives
\[
\aligned
1<S\big(W^S_{\bullet,\bullet};\mathscr{C}\big)&=\frac{1}{10}\left\{\varint_{0}^{1}\varint_{0}^{1}\big(P(u)\big\vert_{S}-v\mathscr{C}\big)^2dvdu+\varint_{0}^{1}\varint_{1}^{\frac{4-u}{2}}\big((4-u-2v)\ell\big)^2dvdu+\right.\\
&\pushright{\left.+\varint_{1}^{2}\varint_{0}^{2-u}\big(P(u)\big\vert_{S}-v\mathscr{C}\big)^2dvdu+\varint_{1}^{2}\varint_{2-u}^{\frac{6-3u}{2}}\big((6-3u-2v)\ell\big)^2dvdu\right\}}\\
&=\frac{1}{10}\left\{\varint_{0}^{1}\varint_{0}^{1} (4-u-2v)^2-4(1-v)dvdu+\varint_{0}^{1}\varint_{1}^{\frac{4-u}{2}}\big(4-u-2v\big)^2dvdu+\right.\\
&\pushright{\phantom{s}\left.+\varint_{1}^{2}\varint_{0}^{2-u}(6-3u-2v)^2-4(2-u-v)dvdu+\varint_{1}^{2}\varint_{2-u}^{\frac{6-3u}{2}}(6-3u-2v)^2dvdu\right\}}\\
&=\frac{23}{40},
\endaligned
\]
which is a contradiction. This completes the proof of Proposition.
\end{proof}

\begin{corollary}
Both $Z$ and $\pi(Z)$ are~irreducible curves, and $\pi(Z)$ is not entirely contained in $\Pi_1\cup\Pi_2\cup\Pi_3\cup\Pi_4\cup Q$.
\end{corollary}

Using  \cite[Lemma~1.4.4]{ACCFKMGSSV}, we see that $\alpha_{G,Z}(X)<\frac{3}{4}$.
Now, using \cite[Lemma~1.4.1]{ACCFKMGSSV}, we see that there are a~$G$-invariant effective $\mathbb{Q}$-divisor $D$ on the~threefold $X$
and a~positive rational number $\mu<\frac{3}{4}$ such that $D\sim_{\mathbb{Q}}-K_X$ and $Z$ is contained in the~locus $\mathrm{Nklt}(X,\mu D)$.
Moreover, it follows from Claim that $\mathrm{Nklt}(X,\mu D)$ does not contain $G$-irreducible surfaces
except maybe for $\widetilde{Q}$, $H_1$, $H_2$, $H_3$, $H_4$.
Now, applying  \cite[Corollary A.1.13]{ACCFKMGSSV} to $(\mathbb{P}^3,\mu\pi(D))$, we see that $\pi(Z)$ must be a $G$-invariant line in $\mathbb{P}^3$.
But this is impossible by Corollary, since all $G$-invariant lines in $\mathbb{P}^3$ are contained in $\Pi_1\cup\Pi_2\cup\Pi_3\cup\Pi_4$.

The obtained contradiction completes the proof of our Theorem.

\medskip

\textbf{Acknowledgements.}
Cheltsov has been supported by EPSRC Grant EP/V054597/1. Park has been supported by IBS-R003-D1, Institute for Basic Science in Korea.


\begin{thebibliography}{3}
\bibitem{AbbanZhuang}
H.~Abban, Z.~Zhuang, \emph{K-stability of Fano varieties via admissible flags}, arXiv:2003.13788 (2020).

\bibitem{ACCFKMGSSV}
C.~Araujo, A.-M.~Castravet, I.~Cheltsov, K.~Fujita, A.-S.~Kaloghiros, J.~Martinez-Garcia, C.~Shramov, H.~S\"u\ss, N.~Viswanathan,
\emph{The Calabi problem for Fano threefolds}, MPIM preprint 2021-31.

\bibitem{Zhuang}
Z.~Zhuang, \emph{Optimal destabilizing centers and equivariant K-stability},  Invent. Math. \textbf{226} (2021), 195--223.
\end{thebibliography}
\end{document}